\documentclass[11pt]{amsart}

\usepackage{color}
\usepackage{amsmath}
\usepackage{amssymb}
\usepackage{amsthm}
\usepackage{enumerate}
\usepackage{amsbsy}
\usepackage{amsfonts}
\newtheorem{theo}{Theorem}[section]

\newtheorem{prop}[theo]{Proposition}

\newtheorem{lemm}[theo]{Lemma}

\newcommand{\al}{\alpha}

\newcommand{\Ga}{\Gamma}

\newcommand{\om}{\omega}
\newcommand{\Om}{\Omega}
\newcommand{\ka}{\kappa}
\newcommand{\si}{\sigma}

\newcommand{\De}{\Delta}
\newcommand{\de}{\delta}

\newcommand{\pa}{\partial}

\newcommand{\R}{{\bf R}^n}

\newcommand{\ri}{\rightarrow}
\newcommand{\Rn}{{\bf R}^{n-1}}

\newcommand{\na}{\nabla}

\begin{document}
\baselineskip=18pt

\title[Maximum modulus for Stokes equations]{Quasi-maximum modulus principle for the Stokes equations}
\author{TongKeun Chang and Hi Jun Choe}
\address{TongKeun chang: Department of
Mathematics, Yonsei University, 50 Yonsei-ro, Seodaemun-gu, Seoul,
South Korea 120-749 } \email{chang7357@yonsei.ac.kr }
\address{Hi Jun Choe: Department of
Mathematics, Yonsei University, 50 Yonsei-ro, Seodaemun-gu, Seoul,
South Korea 120-749 } \email{choe@yonsei.ac.kr }

\thanks{}

\begin{abstract}
In this paper, we extend the maximum modulus estimate  of the
solutions of the nonstationary Stokes equations in the bounded $C^2$
cylinders for the space variables  in \cite{CC} to time estimate. We
show that if the boundary data is $L^\infty$ and the normal part of
the boundary data has log-Dini continuity with respect to time, then
the velocity is bounded. We emphasize that there is no continuity
assumption on space variables in the new
maximum modulus estimate. This completes the maximum modulus estimate.\\



\end{abstract}

\maketitle

\section{Introduction}
\setcounter{equation}{0}

In this paper, we study the maximum modulus principle of the
nonstationary Stokes equations:
\begin{align}\label{maineq}
\begin{array}{ll}\vspace{2mm}
u_t -  \De u + \na p =0 &  \mbox{ in } \Omega \times (0,T),\\
\vspace{2mm} div \, u =0 &  \mbox{ in } \Omega\times (0,T),\\
\vspace{2mm}
u|_{t=0} = 0 &  \mbox{ in } \Omega, \\
u|_{\pa \Om \times (0,T)}= g   & \mbox{ on } \partial\Omega \times
(0,T),
\end{array}
\end{align}
where $\Om $ is $C^2$ bounded  domain in $\R ( n\geq 3)$ such that the boundary $\pa \Om$ of $\Om$ is   connected and
$0<T \leq \infty$. We
assume the boundary data $g$ satisfies the compatibility condition:
$$
\int_{\partial\Om} g(Q,t)\cdot N(Q) dQ = 0
$$
for almost all $t \in (0,T)$, where $N(Q)$ is the outward unit
normal vector at $Q\in \pa \Om$.

The maximum modulus principle for the stationary Stokes
equations were studied by many mathematicians  (see \cite{Ch},
\cite{KR}, \cite{M}, \cite{MRu}, \cite{Sh2}, \cite{V}, etc).

But, the maximum
modulus estimate of the velocity of the nonstationary Stokes
equations, when the boundary data is bounded, are not known.
A. V. Solonnikov\cite{So2} showed that if $\Om$ is $C^{2 + \al}, 0 < \al$ smooth convex bounded domain and $g \in C(\pa \Om \times (0,T))$ with
$g\cdot N =0$, the solution $u$ of \eqref{maineq} is continuous in $\Om \times (0,T)$ such that
\begin{align*}
\sup_{(x,t) \in \Om \times (0,T)} |u(x,t)| \leq c \sup_{(Q,t) \in \pa \Om \times (0,T)} | g(Q,t)|
\end{align*}
for some positive constant $c$ independent of $g$.
%In
%fact, we think it is false when the boundary data has a peculiar
%discontinuity in both time and space.
In recent, T.
Chang and H. Choe\cite{CC} improve the A. V. Solonnikov's result.
%obtained a   maximum principle of
%nonstationary Stokes equations \eqref{maineq}.
For the self
contained presentation and rigorous expressions, we repeat the same
symbols and definitions of \cite{CC}.
 We denote $E$ for the fundamental solution to Laplace equation and $\Ga$
for the fundamental solution to heat equation with unit
conductivity.  We define the $(n-1)$-dimensional convolution
$$
{\bf S} (f)(x) = \int_{\partial \Omega} E(x-Q) f(Q) dQ, \qquad x \in \Om
$$
for a real-value function $f:\partial\Omega \ri {\bf R}$ which is
just the single layer potential of $f$ on $\partial\Omega$. We
introduce a composite kernel which is the core of Poisson kernel. We
define a composite kernel function $\kappa(x,t)$ on $\Om\times
(0,\infty)$ by
\begin{align*}
\kappa (x,t) = & \int_{\pa \Om}    \frac{\partial}{\partial
N(Q)}\Ga(x-Q,t)    E (Q)dQ
\end{align*}
and a surface  potential ${\bf  T}$ for $f$ by
$$
{\bf T}(f)(x,t) =4 \int_{0}^{t }\int_{ \partial\Om} \kappa(x-Q,t-s)
f(Q,s) dQ ds, \qquad (x,t) \in \Om \times (0,T)
$$
for real-value function $f:\partial\Omega\times (0,\infty) \ri {\bf
R}$.

For given $x\in \Om$, $P_x$ is the nearest point of $x$ on
$\partial\Om$ such that $dist(x,\partial\Om)=|P_x-x|$ and for a
vector valued function $v(x)$, we define the normal component $v_N$
and tangential component $v_T$ to the nearest point $P_x$ by
$$
v_N(x) = (v(x)\cdot N(P_x)) N(P_x) \quad \mbox{and}\quad v_T(x)=
v(x) - v_N(x).
$$
In\cite{CC}, the essential estimate is stated as the following
proposition:
\begin{prop}\label{theorem1}
Suppose that $\Omega$ is a  bounded $C^2$ domain and $u$ is the
solution to (\ref{maineq}) for the bounded boundary data $g$. The
normal component $u_N$ of $u$ is bounded and  there is also a
constant $C(\Om)$ such that
$$
\max_{(x,t)\in \Omega\times (0,T)} |u_N (x,t)|
 \leq C(\Omega) \| g\|_{L^\infty (\pa \Om \times (0,T))}.
% \max_{(y,t)\in \partial\Omega\times (0,T) } |g(y,t)|.
$$
Furthermore, the tangential component $u_T$ of the velocity $u$
satisfies that
$$
\max_{(x,t)\in \Omega\times (0,T)} |u_T(x,t)-\nabla {\bf S}(g \cdot
N)_T (x,t) -\nabla {\bf T}(g \cdot N)_T (x,t) |
 \leq C(\Omega)   \| g\|_{L^\infty (\pa \Om \times (0,T))}.
 %\max_{(y,t)\in \partial\Omega\times (0,T) } |g(y,t)|
$$

\end{prop}
For the proof of  Proposition \ref{theorem1} in ${\bf R}^n_+ \times
(0,T)$ in \cite{CC}, the authors used the Riesz operator property of
Poisson kernel matrix of Stokes equations obtained by V. A.
Solonnikov (see section \ref{Poissonkernel}). The velocity $u$ is
represented by convolution of Poisson kernel matrix $(K_{ij})_{1
\leq i,j \leq n}$ and boundary data $g$. For $1 \leq i \leq n-1$,
the Poisson kernel $K_{in}$ contains the term
 $\de(t)  D_{x_i} E(x'-y',x_n) + B_{in}(x,t)$
 which is a kernel of integral operator
 $ \na {\bf S}(g_n)_T+ \nabla {\bf T}(g_n)_T $, where
$\de(t)$ is Dirac delta function of time,  $g_n$ is the $n$-th
element of
 $g =(g_1, \cdots, g_n)$ which is the normal component and
\begin{align*}
 B_{in}(x,t) =
 -\int_{{\bf R}^{n-1}}D_{x_n}
\Ga(x^\prime -y^\prime, x_n, t) D_{y_i} E( y^{\prime},0) dy^\prime.
\end{align*}
It is a new observation in this paper that there is a cancelation
between $\na {\bf S}(g_n)$ and $\nabla {\bf T}(g_n)$ as an operator
of BMO(functions of bounded mean oscillation). Nonetheless, the
potential $ \na {\bf S}(g_n)_T+ \nabla {\bf T}(g_n)_T $ could bow up
in $L^\infty$. In fact, we can show that there exists $g_n \in L^p \cap
L^\infty, \, 1 \leq p < \infty$ such that it blows up.
This implies that Proposition \ref{theorem1} is optimal.
\begin{theo}\label{c-example}
 There exists $g \in L^p \cap
L^\infty, \, 1 \leq p < \infty$ such that the velocity $u$ blows up.
\end{theo}

From the kernel   $ \de(t)  D_{x_i} E(x'-y',x_n) + B_{in}(x,t)$, if $g_n$ has  Dini-continuity with respect to space, that is,
 $|| f||_{Dini,\pa \Om} = \int_{0}^{r_0} \sup_{x\in \pa \Om} \om(f)(r,x) \frac{dr}{r} < \infty$,
where $\om(f)$ is the modulus of continuity of $f$ such that
$\om(f)(r,x)=\sup_{y\in B_r(x)\cap \pa \Om} |f(y)-f(x)|$, then
integral operators $(\na{\bf S})_T(f)$ and $(\na {\bf T})_T(f)$ are
bounded. This is mainly due to boundedness of Riesz operator in BMO.
Using this fact, it is obtained the following result:
\begin{prop}\label{theorem2}
Suppose that the domain $\Om$ is bounded $C^2$ and $u$ is a solution
to (\ref{maineq}). Suppose $g$ is bounded on $\partial\Om \times
(0,T)$ and the normal component $g_N$ is Dini-continuous with
respect to space. Then, there is a constant $C(\Om)$ depending only
on $\Om$ such that
\begin{align*}
 \max_{(x,t)\in \Omega\times (0,T) } |u(x,t) |
 \leq C ( \Omega) \Big( ( \| g\|_{L^\infty (\pa \Om \times (0,T))}
+ \sup_{t \in (0,T)} \|g_N(\cdot, t)||_{Dini,\partial\Omega} \Big).
\end{align*}
\end{prop}

Proposition \ref{theorem2} means that the maximum modulus principle of the Stokes equations \eqref{maineq} can be obtained by
the Dini-continuity with respect to the space only. It is important
to clarify that the maximum modulus of solution remains bounded
under the time continuity of boundary data without continuity
assumption in space. In this paper, we answer to the question, in
fact, we find a cancelation between $\de(t)  D_{x_i} E(x'-y',x_n)$
and $B_{in}(x,t)$ in logDini-continuous functions in time. We say
that $f :{\bf R} \ri {\bf R}$ is logDini-continuous if
 $|| f||_{logDini,(0,T)}:=  \int_{0}^{r_0} \sup_{t \in (0,T)} \om(f)(r,t) \frac{|\ln r|}{r} dr < \infty$,
where $\om(f)(r,t)=\sup_{s\in (t-r, t+r) \cap (0,T)} |f(s)-f(t)|$.
To handle the initial time $t=0$, we extend $g_N$ to ${\bf R}$ by
$$
g_N(x,t) =0 \quad\mbox{for} \quad x\in \partial\Omega \quad t\leq 0.
$$
\begin{theo}\label{s-t-theorem}
Suppose that the domain $\Omega$ is bounded $C^2$ and $u$ is a
solution to (\ref{maineq}) for bounded boundary data $g$. Suppose
that $g_N$ is logDini-continuous with respect to time. Then,
\begin{align*}
& \max_{(x,t)\in \Omega\times (0,T)} |\nabla {\bf S}(g_N)_T (x,t)
+\nabla {\bf T}(g_N)_T (x,t) |\\
& \hspace{10mm}  \leq C(\Omega) \Big(\| g_N \|_{L^\infty(\pa \Om
\times (0,T))} +  \sup_{x \in \pa \Om}\| g_N(x,
\cdot)\|_{logDini}\Big).
\end{align*}
\end{theo}

With Proposition \ref{theorem1}, we get
\begin{theo}
Suppose that the domain $\Omega$ is bounded $C^2$ and $u$ is a
solution to (\ref{maineq}) for bounded boundary data $g$. Suppose
that $g_N$ is logDini-continuous with respect to time. Then,  the solution  $u$ of \eqref{maineq}
satisfies
\begin{align*}
 \|u \|_{L^\infty (\Omega\times (0,T))}
 \leq C(\Omega) \Big(\| g\|_{L^\infty(\pa \Om \times (0,T))}
 +   \sup_{Q \in \pa \Om} \| g_N(Q, \cdot)\|_{logDini}\Big).
\end{align*}

\end{theo}

We present the paper in the following way. In section 2, we discuss
about the Poisson kernel of the nonstationary Stokes equations in
$\R_+ \times (0,T)$. In section 3, we prove the main Theorem
\ref{s-t-theorem} in the half space. In section 4, we prove the
Theorem \ref{s-t-theorem}. In section 5, we prove the Theorem \ref{c-example}.

\section{Kernels on  half plane}\label{Poissonkernel}
\setcounter{equation}{0}
For notation,  we denote $x=(x^{\prime},x_n)$,
that is, $x^{\prime}=(x_1,x_2,\cdot\cdot\cdot,x_{n-1})$. Indeed, the symbol $ {\prime}$
means the coordinate up to $n-1$ and
 $\omega_n$ is the surface area of the unit sphere in ${\bf R}^n$. We also denote that $D_{x_i} u$
  are partial derivatives of $u$ with respect to $x_i, \,\, 1 \leq i \leq n$, that is, $D_{x_i} u(x) =
\frac{\pa }{\pa x_i}u(x)$.

 We let $E$ and $\Gamma$  be the fundamental solutions to the Laplace equation
 and the heat equation, respectively, such that
\begin{align*}
E(x,t) = -\frac{1}{(n-2)\om_n} \frac{1}{|x|^{n-2}}, \qquad
\Gamma(x,t)  = \left\{\begin{array}{ll} \vspace{2mm}
\frac{1}{\sqrt{2\pi t}^n} e^{ -\frac{|x|^2}{2t} }, &   t > 0\\
  0, & t < 0,
\end{array}
\right.
\end{align*}
where $\om_n$ is the measure of the unit sphere in $\R$.

The Poisson kernel $(K,\pi ) $ for the half space is defined by
\begin{align*}
\begin{array}{ll}\vspace{2mm}
K_{ij}(x'-y',x_n,t) &  =  -2 \delta_{ij} D_{x_n}\Ga(x'-y', x_n,t)  -L_{ij} (x'-y',x_n,t) \\ \vspace{2mm}
& \quad +  \de_{jn} \de(t)  D_{x_i} E(x'-y',x_n),\\ \vspace{2mm}
\pi_j (x'-y',x_n,t )&=-2\delta(t) D_{x_j}D_{x_n}E(x'-y',x_n)+4 D_{x_n}D_{x_n}A(x'-y',x_n,t)\\
&\quad +4D_{t} D_{x_j}A(x'-y',x_n,t),
\end{array}
\end{align*}
where $\delta(t)$ is the Dirac delta function and $\delta_{ij}$ is the Kronecker delta function.
Here, we defined  that
\begin{align*}
 {L}_{ij} (x,t) & =  D_{x_j}\int_0^{x_n} \int_{\Rn}   D_{z_n}    \Ga(z,t) D_{x_i}   E(x-z)  dz,\\
 A(x,t)&=\int_{\Rn}\Ga(z^{\prime},0,t)E(x^{\prime}-z^{\prime},x_n)dz^{\prime}.
\end{align*}

We have relations among $L$ such that
\begin{align} \label{1006-3}
\sum_{1 \leq i \leq n} L_{ii} = -2D_{x_n} \Ga, \qquad  L_{in} = L_{ni}  + B_{in},
\end{align}
where
\begin{align}
 B_{in}(x,t) = \left\{\begin{array}{cl} \vspace{2mm}
 -\int_{{\bf R}^{n-1}}D_{x_n}
\Ga(x^\prime -y^\prime , x_n, t) D_{y_i} E( y^{\prime},0) dy^\prime =  \frac{\partial}{x_i} \kappa (x,t)
 & \quad  i\neq n,\\
0 & \quad i = n
\end{array}
\right.
\end{align}
 (see \cite{K} and \cite{So}).

The solution $(u, p)$ of the Stokes system \eqref{maineq} in $\Om = {\bf R}^n_+$ with boundary data
$g$ is expressed by
\begin{align}\label{representation}
\begin{array}{ll}\vspace{2mm}
u^i(x,t) &= \sum_{j=1}^{n}\int_0^t \int_{\Rn}
K_{ij}( x^{\prime}-y^{\prime},x_n,t-s)g_j(y^{\prime},s) dy^{\prime}ds,\\
p(x,t) &= \sum_{j=1}^{n}\int_0^t \int_{\Rn} \pi_j(x^{\prime}-y^{\prime},x_n,t-s) g_j(y^{\prime},s)
dy^{\prime}ds.
\end{array}
\end{align}
(See \cite{K} and \cite{So}).

\section{Maximum modulus estimate in the half space}
\setcounter{equation}{0}
In this section, we consider the maximum modulus estimate in the half space.
In \cite{CC}, the authors proved the following proposition.
\begin{prop}\label{L-1-norm}
Let $1 \leq i \leq n$ and $1 \leq j \leq n-1$. Then
\begin{align}\label{0126-1}
\int_0^\infty \int_{{\bf R}^{n-1}} |L_{ij}(x',x_n,t)| dx'dt < C,
\end{align}
where $C>0$ is independent of $x_n >0$ and hence  it follows that
\begin{align}\label{0126-2}
\int_0^\infty \int_{{\bf R}^{n-1}} |L_{in}(x',x_n,t) - B_{in}(x^\prime,x_n,t)| dx'dt < C,
\end{align}
where $C>0$ is independent of $x_n >0$.
\end{prop}
The maximum modulus theorem for the half space follows from proposition \ref{L-1-norm} (see \cite{CC}).

\begin{prop}\label{propCC}
Let $g= (g_1, g_2, \cdots,  g_n) \in L^\infty(\Rn \times ( 0, T ))$ and $(u,p)$ is
represented by \eqref{representation}. Then,
\begin{align}
\| u_T - \nabla {\bf S}_T (g_n) - \nabla {\bf T}_T (g_n) \|_{L^\infty (\R_+ \times (0, T))}
\leq C \| g\|_{L^\infty(\partial{\bf R}^{n}_{+} \times (0,T))}
\end{align}
for some $C> 0$ independent of $T$. Furthermore, the normal
component of the velocity $u$ is bounded and  there is also a
constant  $C$ such that
$$
\|u_n\|_{L^\infty({\bf R}^n_+\times (0,T))}
 \leq C  \|g\|_{L^\infty(\pa {\bf R}^n_+\times (0,T))}.
$$
\end{prop}

Now, to prove the main theorem \ref{s-t-theorem}, we prove the following lemma.
\begin{lemm}\label{lemma0911}
Let $g_n  \in L^\infty(\Rn \times ( 0, T ))$ and
 $g_n$ satisfies the logDini-continuity with
 $supp \, g_n \subset B'_{ M} =\{|x'| <   M \}$
for some $M>0$, then
\begin{align*}
&\|(\na {\bf S}_T
       +\na {\bf T}_T)(g_n)\|_{L^\infty({\bf R}^n_+  \times (0,T))}\\
& \leq C_M\Big(\| g_n\|_{L^\infty(\Rn \times (0,T)}
   + \| \max_{x' \in {\bf R}^{n-1}} g_n(x',\cdot)\|_{logDini}
   \Big)
\end{align*}
for some $C_M> 0$.
\end{lemm}
By Proposition \ref{propCC} and lemma \ref{lemma0911}, we obtain
\begin{theo}\label{maximum modulus-Rn}
Let $g= (g_1, g_2, \cdots,  g_n) \in L^\infty(\Rn \times ( 0, T ))$ and
$g_n$ satisfies the logDini-continuity with $supp \, g_n(\cdot, t) \subset B'_{  M} = \{ x' \in \Rn \, | \, |x'| < M\}$ for all $0 < t < \infty$.
Then the $u$
represented by \eqref{representation} satisfies
\begin{align*}
&\| u_T \|_{L^\infty ({\bf R}^n_+  \times (0, T))}\\
&  \leq C_M \Big(\| g_n\|_{L^\infty(\Rn \times (0,T)}
   + \|\max_{x' \in {\bf R}^{n-1}} g_n(x',\cdot)\|_{logDini}
  \Big)
\end{align*}
for some $C_M> 0$.
\end{theo}

To show the lemma \ref{lemma0911}, we study the kernel of $(\nabla
{\bf S}_T +\nabla {\bf T}_T)(x, y', t,s)$. Note that the kernel of $(\nabla {\bf S}_T +\nabla {\bf
T}_T)_i $   is $\de(t)  D_{x_i} E(x'-y',x_n) + B_{in}(x,t)$. Note
that
\begin{align} \label{0920}
 B_{in}(x,t) &= -D_{x_n}\Ga_1(x_n,t) \int_{{\bf R}^{n-1}}
              \Ga'(y^\prime, t) D_{y_i} E( x'- y^{\prime},0) dy^\prime \quad i \neq n,
\end{align}
where  $\Ga_1$ and $\Ga^{\prime}$ are Gaussian kernels in ${\bf R}$ and $\Rn$, respectively.

\begin{lemm}\label{120906}
For $1 \leq i \leq n-1$, we get
\begin{align}\label{0119-1}
\begin{array}{l}  \vspace{2mm}
|\int_{|y'| \leq \frac12 |x'|} \Ga^{\prime} (y',t)    D_{y_i} E(x'-y',0)  dy'|
 \leq C t^{-\frac{n-1}2} e^{-c\frac{|x'|^2}t }
         +C  |x' |^{-n+1}
           \int_{| y'| \leq \frac12\frac{ |x'|}{\sqrt{t}}}  |y'|^2 e^{-c|y'|^2}
             dy'\\ \vspace{2mm}
|\int_{\frac12 |x'| \leq |y'| \leq 2|x'|, |x'-y'| \geq \frac12 |x'|}
                 \Ga^{\prime} (y',t)   D_{y_i} E(x'-y',0) dy'|
 \leq C t^{-\frac{n}2 +\frac12}  e^{-c\frac{|x'|^2}{t}},\\ \vspace{2mm}
|\int_{|x'-y'| \leq \frac12 |x'|}\Ga^{\prime} (y',t)   D_{y_i} E(x'-y',0) dy'|
\leq  C t^{-\frac{n+1}2} |x'|^2 e^{-c\frac{|x'|^2}t}\\
             |\int_{|y'| \geq 2 |x'|} \Ga^{\prime} (y',t)   D_{y_i}  E(x'-y',0)  dy'|
\leq C  t^{\frac{-n+1}2  } \int_{\frac{2 |x'|}{\sqrt{t}} \leq |y'| }
           |y'|^{-n +1}  e^{-c|y'|^2} dy',
\end{array}
\end{align}
where $c, \,\,C> 0$ are independent of $x', y_n$ and $t$.
\end{lemm}
\begin{proof}
Using integration by parts, we get
\begin{align}\label{0919}
& \notag \int_{|y'| \leq \frac12 |x'|}\Ga^{\prime} (y',t)   D_{y_i} E(x'-y',0)  dy'\\
 & \quad  =  \int_{|y'| = \frac12 |x'|}  \frac{ y_i}{|y'|}\Ga^{\prime} (y',t)
          E(x'-y',0) \si( dy')
              -\int_{|y'| \leq \frac12 |x'|}  D_{y_i}\Ga^{\prime} (y',t) E(x'-y',0)  dy'.
\end{align}
For $y'$ with $| y'| = \frac12 |x'|$, we get
$ | \Ga^{\prime} (y',t)  | \leq C t^{-\frac{n-1}2} e^{-c\frac{|x'|^2}{t}}$ and
$|E(x'-y',0)|  \leq C \frac{1}{|x'|^{n-2}}$.
Here, the first term of the right hand side in \eqref{0919} is dominated by
\begin{align}\label{0119}
\begin{array}{ll}\vspace{2mm}
  \int_{|y'| =\frac12 |x'|}    |\Ga^{\prime}  (y',t)| |  E(x'-y',0)| \si (  dy')
 &\leq C t^{-\frac{n-1}2} e^{-c\frac{|x'|^2}t } \frac{|x'|^{n-2}}{|x'|^{n-2}}\\
 & \leq C t^{-\frac{n-1}2} e^{-c\frac{|x'|^2}t }.
 \end{array}
 \end{align}
 Since $ \int_{| y'| \leq \frac12 |x'|}    D_{y_i }  \Ga^{\prime} (y',t)   dy' =0$
 for all $t > 0$,
 using the Mean-value theorem,  the second term of  the right hand side of \eqref{0919} satisfies
\begin{align} \label{0119-2}
\begin{array}{ll} \vspace{2mm}
%&\int_{| y'| \leq \frac12 |x'|} \Ga^{\prime} (y',t) D_{y_i }  E(x'-y',0)   dy'\\ \vspace{2mm}
 &  |\int_{| y'| \leq \frac12 |x'|} D_{y_i} \Ga^{\prime} (y',t) (
             D_{y_i} E(x'-y',0) - D_{y_i} E(x',0) )  dy'|\\ \vspace{2mm}
& \leq C  |x' |^{-n+1}
           \int_{| y'| \leq \frac12 |x'|} t^{-\frac{n-1}2 -1} |y'|^2 e^{-c\frac{|y'|^2}t}
             dy'\\
           \vspace{2mm}
& \leq C  |x' |^{-n+1}
           \int_{| y'| \leq \frac12\frac{ |x'|}{\sqrt{t}}}  |y'|^2 e^{-c|y'|^2}
             dy'.
\end{array}
\end{align}
By \eqref{0919} - \eqref{0119-2}, we obtain $\eqref{0119-1}_1$.

For  $\eqref{0119-1}_2$, note that for $y'$ satisfying $\frac12 |x'|
\leq |y'| \leq 2|x'|$ we have $ |x'-y'| \geq \frac12 |x'|$.  We have
$| \Ga^{\prime} (y',t)| \leq C t^{-c\frac{n-1}2  }
e^{-c\frac{|x'|^2}{t}}$ and $ D_{y_i} E(x'-y',0) \leq  C
|x'|^{-n+1} $, and thus we get
 \begin{align*}
|\int_{\frac12 |x'| \leq |y'| \leq 2|x'|, |x'-y'| \geq \frac12 |x'|}
    \Ga^{\prime} (y',t) D_{y_i} E(x'-y',0)  dy'|
& \leq C  t^{-\frac{n}2 +\frac12}     e^{-c\frac{|x'|^2}{t}}.
\end{align*}
Hence, we obtain  $\eqref{0119-1}_2$.

For $y'$ satisfying $|x' -y'| \leq \frac12|x'|$, we have
$|\na' \Ga'(y', t)| \leq C t^{-\frac12 n -\frac12} |x'| e^{-c\frac{|x'|^2}t}$.
Hence, since $\int_{|x'- y'| \leq \frac12 |x'|}  D_{y_i}E(x' -y', 0)    dy  =0$,
using Mean-value Theorem $ \eqref{0119-1}_3$ is proved by
\begin{align*}
 & |\int_{|x'-y'| \leq \frac12 |x'|} \Big(\Ga^{\prime} (y',t) - \Ga^{\prime} (x',t) \Big)
               D_{y_i} E(x'-y',0)  dy'|\\
& \leq C t^{-\frac{n+1}2} |x'| e^{-\frac{|x'|^2}t} \int_{|x'-y'| \leq \frac12|x'|}
 |x'-y'|^{-n +2}dy' \\
& \leq C t^{-\frac{n+1}2} |x'|^2 e^{-c\frac{|x'|^2}t}.
\end{align*}
Finally, $\eqref{0119-1}_4$  follows by
\begin{align*}
|\int_{|y'| \geq 2 |x'|}\Ga^{\prime} (y',t)  D_{y_i}E(x'-y',0)  dy'|
 &\leq C t^{-\frac{n -1}2  } \int_{ 2 |x'|  \leq |y'| }
          |y'|^{-n +1}  e^{-c\frac{|y'|^2}{t}} dy'\\
 & =C t^{\frac{-n+1}2  } \int_{\frac{2 |x'|}{\sqrt{t}} \leq |y'| }
           |y'|^{-n +1}  e^{-c|y'|^2} dy'.
\end{align*}
\end{proof}

\begin{lemm}\label{120907-2}
Let $ M > 0$. Then, for $x' \in {\bf R}^{n-1}$ and $s > 0$, we get
\begin{align*}
   \int_{|y'| \leq  M }|\int_{\Rn}   \Ga'(x'-z', s) D_{z_i} E(y' -z',0) dz'|dy'
 \leq C ( 1 + a(M,s)),
\end{align*}
where $a(M,s) : = \ln M -  \ln{\min{( \sqrt{s},M) }}$.
\end{lemm}
\begin{proof}
Using the change of variables, we have
\begin{align}\label{L-10917}
& \int_{|y'| \leq  M }|\int_{\Rn}   \Ga'(x'-z', s) D_{z_i} E(y' -z',0) dz'|dy'\\
& \notag\qquad  =  \int_{|x' - y'| \leq  M }|\int_{\Rn}   \Ga'(z', s) D_{z_i} E(y' -z',0) dz'|dy'.
\end{align}
For fixed $y'$ satisfying $|x' -y'| \leq M$, let us
\begin{align*}
A_1 &= \{ z' \in \Rn \, | \, |z'| \leq \frac12 |y'| \},\\
A_2 &=\{ z' \in \Rn \, | \, |y' - z'| \leq \frac12 |y'| \},\\
A_3 &=\{ z' \in \Rn \, | \,\frac12 |y'| \leq  |z'| \leq 2|z'|, \,\, |y' - z'| \geq \frac12 |y'|  \},\\
A_4 &=\{ z' \in \Rn \, | \,|y' - z'| \geq 2 |y'| \}.
\end{align*}
We divide the integral of the left-side of \eqref{L-10917} by four parts;
\begin{align*}
 \int_{|x'-y'| \leq   M }|\int_{\Rn} \Ga'(z', s) D_{z_i} E(y' -z',0) dz'|dy'
 = I_1 + I_2 + I_3 + I_4,
\end{align*}
where
\begin{align*}
I_1 &= \int_{|x'-y'| \leq   M }|\int_{A_1}  \Ga'(z', s) D_{z_i} E(y' -z',0) dz'|dy',\\
I_2 &= \int_{|x'-y'| \leq    M }|\int_{A_2} \Ga'(z', s) D_{z_i} E(y' -z',0) dz'|dy',\\
I_3 &=  \int_{|x'-y'| \leq    M }|\int_{A_3}  \Ga'(z', s) D_{z_i} E(y' -z',0) dz'|dy',\\
I_4 &= \int_{x'-|y'| \leq    M }|\int_{A_4} \Ga'(z', s) D_{z_i} E(y' -z',0) dz'|dy'.
\end{align*}
By $\eqref{120906}_1$ of the lemma \ref{120906}, we get
\begin{align}\label{130117-1}
I_1 & \leq  C \int_{|x'-y'| \leq  M }
        \Big(   s^{-\frac{n}2 +\frac12} e^{-c\frac{|y'|^2}s} +
      |y'|^{-n+1}  \int_{|z'| \leq \frac12\frac{|y'|}{\sqrt{s}}}|z'|^2 e^{-|z'|^2} dz' \Big)dy'\\
    & \notag : = I_{11} + I_{12}.
\end{align}
Here,
\begin{align}\label{130118-1}
I_{11} & \leq C  \int_{{\bf R}^{n-1}}
           s^{-\frac{n}2 +\frac12} e^{-c\frac{|y'|^2}s} dy'\\
\notag  &    = C.
\end{align}
And
\begin{align}\label{130118-2}
    &   I_{12} \leq  C\int_{|x'-y'| \leq  M, \, \frac12\sqrt{s} \geq |y'| }
                  |y'|^{-n+1} ( \frac{|y'|}{\sqrt{s}})^{n+1} dy'
                  +  C\int_{|x'-y'| \leq  M, \, \frac12\sqrt{s} \leq |y'| }|y'|^{-n+1}dy'.
\end{align}
If $|x'| \leq 2M$, then from \eqref{130118-2}, we have
\begin{align}\label{130118-}
    I_{12} & \leq   C\int_{ |y'| \leq  \min{(\frac12 \sqrt{s},3M) } }|y'|^{-n+1}
               (\frac{|y'|}{\sqrt{s}})^{n+1}dy'
      +  C\int_{\min{(\frac12 \sqrt{s},\frac12 M) } < |y'| \leq 2 M }|y'|^{-n+1} dy'\\
    & \notag \leq   C  ( 1 + a(M,s)).
\end{align}

If $|x'| \geq 2M$ and $s \geq (3|x'|)^2$, then from
\eqref{130118-2}, we have
\begin{align}\label{130118-4}
I_{12}  & \leq C\int_{|x'-y'| \leq  M, \, \frac12\sqrt{s} \geq |y'|}
                  |y'|^{-n+1} ( \frac{|y'|}{\sqrt{s}})^{n+1} dy'\\
    & \notag \leq C|x'|^2 s^{-\frac{n}2 -\frac12} \int_{|x'-y'| \leq  M } dy'\\
    & \notag \leq C.
\end{align}

If $|x'| \geq 2M$ and $s \leq (\frac12|x'|)^2$, then from
\eqref{130118-2}, we have
\begin{align} \label{130118-5}
I_{12}  & \leq  C|x'|^{-n+1}\int_{|x'-y'| \leq  M}dy'\\
    &  \notag \leq C.
\end{align}

If $|x'| \geq 2M$ and $ (\frac12|x'|)^2 \leq s \leq (3|x'|)^2$, then from \eqref{130117-1}, we have
\begin{align} \label{130118-6}
I_{12}  &   \leq  C|x'|^2 s^{-\frac{n}2 -\frac12}
                  \int_{|x'-y'| \leq  M }dy'
                  + C|x'|^{-n+1} \int_{|x'-y'| \leq  M } dy'\\
    & \notag \leq C.
\end{align}
%From \eqref{130118-2}-\eqref{130118-6}, we get
%\begin{align}\label{120907}
%I_{12} \leq C ( 1 + a(M,s))
%\end{align}
Hence, by \eqref{130117-1} and \eqref{130118-6}, we have
\begin{align}\label{130204-1}
I_1 \leq C ( 1 + a(M,s)).
\end{align}

By the $\eqref{120906}_2$ and $\eqref{120906}_3$ of  lemma \ref{120906}, we get
\begin{align} \label{120908}
I_2+ I_3 & \leq C\int_{{\bf R}^{n-1} } \Big( s^{-\frac{n-1}2} e^{-c\frac{|y'|^2}s }
         +   s^{-\frac{n}2 -\frac12}|y' |^2 e^{-c\frac{|y'|^2}{s}} \Big) dy'\\
    & \notag \leq C.
\end{align}
By the $\eqref{120906}_4$  of  lemma \ref{120906}, we get
\begin{align}\label{120909}
I_4 & \leq C\int_{|x'-y'| \leq 3M }
          s^{\frac{-n+1}2  }  \int_{\frac{2 |y'|}{\sqrt{s}} \leq |z'| }|z'|^{-n +1}  e^{-|z'|^2} dz' dy'\\
    & \notag \leq C\int_{|y'| \leq  \frac12 \sqrt{s}  }
          s^{\frac{-n+1}2  } dy'  + C\int_{  \frac12 \sqrt{s}  \leq |y'|  }
          s^{\frac{-n+1}2  }    e^{-\frac{|y'|^2}{\sqrt{s}}}  dy'\\
    & \notag \leq C.
\end{align}
Hence, we completed the proof of lemma.

\end{proof}

{\bf Proof of Lemma \ref{lemma0911}}
Note that for $i \neq n$, by \eqref{1006-3}, we have
\begin{align*}
(\na {\bf S}_T +\na {\bf T}_T)_i(g_n)(x,t)  = {\mathcal B}_{in}g_n(x,t)
             - \frac{\pa}{\pa x_i} {\mathcal E}g_n(x,t),
\end{align*}
where
\begin{align*}
{\mathcal E}g_n(x,t)& = \int_{\Rn}  E (x'-y', x_n)g_n(y',t)dy',\\
 {\mathcal B}_{in}g_n(x,t)& = \int_0^t\int_{\Rn}  B_{in}(x'-y', x_n,
 t-s)g_n(y',s) dy'ds.
\end{align*}

Note that
\begin{align*}
& {\mathcal B}_{in}g_n(x,t)
             - \frac{\pa}{\pa x_i} {\mathcal E}g_n(x,t)\\
&=-\int_0^t \int_{\Rn} g_n(y',s) \int_{\Rn} D_{x_n} \Ga(x'-z', x_n, t-s) D_{z_i} E(y' -z',0) dz'dy'ds\\
& \qquad - \int_{\Rn} g_n(y',t)   D_{y_i}  E(x' -y',x_n) dy'\\
& = -\int_0^t\int_{\Rn} (g_n(y',s)- g_n(y',t))\int_{\Rn} D_{x_n} \Ga(x'-z', x_n, t-s) D_{z_i} E(y' -z',0) dz'dy'ds\\
& \qquad - \int_0^t \int_{\Rn} g_n(y',t) \int_{\Rn} D_{x_n} \Ga(x'-z', x_n, t-s) D_{z_i} E(y' -z',0) dz'dy'ds\\
& \qquad - \int_{\Rn} g_n(y',t)   D_{y_i}  E(x' -y',x_n) dy'.
\end{align*}
Since $\int_0^\infty   \Ga(x', x_n, s)ds= -E(x', x_n)$, we have
$\int_0^\infty D_{x_n} \Ga(x', x_n, s)ds= -\frac{\pa }{\pa x_n} E(x', x_n)$.
Furthermore, we have
$D_{ x_n} E(\cdot , x_n) *' D_{x_i} E(\cdot,0) (x') = D_{x_i} E(x', x_n) $, where $*'$ is
a convolution in ${\bf R}^{n-1}$. Hence, we get
\begin{align*}
&{\mathcal B}_{in}g_n(x,t)
             - \frac{\pa}{\pa x_i} {\mathcal E}g_n(x,t)\\
             & = -\int_0^t\int_{\Rn} (g_n(y',s)- g_n(y',t))
           \int_{\Rn} D_{x_n} \Ga(x'-z', x_n, t-s) D_{z_i} E(y' -z',0) dz'dy'ds\\
& \qquad - \int_0^\infty \int_{\Rn} g_n(y',t) \int_{\Rn} D_{x_n} \Ga(x'-z', x_n, s) D_{z_i} E(y' -z',0) dz'dy'ds\\
 & \qquad- \int_t^\infty \int_{\Rn} g_n(y',t) \int_{\Rn} D_{x_n} \Ga(x'-z', x_n, s) D_{z_i} E(y' -z',0) dz'dy'ds\\
 & \qquad- \int_{\Rn} g_n(y',t)   D_{y_i}  E(x' -y',x_n) dy'\\
 & =  \int_0^t\int_{\Rn} (g_n(y',s)- g_n(y',t))
        \int_{\Rn} D_{x_n} \Ga(x'-z', x_n, t-s) D_{z_i} E(y' -z',0) dz'dy'ds\\
& \qquad- \int_t^\infty \int_{\Rn} g_n(y',t) \int_{\Rn} D_{x_n} \Ga(x'-z', x_n, s) D_{z_i} E(y' -z',0) dz'dy'ds.
\end{align*}
Since $supp\, g_n(\cdot, t)\subset B'_{  M}$ for all $t \in (0,\infty)$, we get
\begin{align*}
& |{\mathcal B}_{in}g_n(x,t)
             - \frac{\pa}{\pa x_i} {\mathcal E}g_n(x,t)|\\
& \leq c\int_0^t (t-s)^{-\frac32} x_n e^{-\frac{x_n^2}{(t-s)}}
\int_{|y'| \leq   M} |g_n(y',s)- g_n(y',t)|  \times \\
& \qquad \qquad
|\int_{\Rn} \Ga'(x'-z', t-s) D_{z_i} E(y' -z',0) dz'|dy'ds\\
& \qquad+ c\int_t^\infty s^{-\frac32} x_n e^{-\frac{x_n^2}s}
        \int_{|y'| \leq   M} |  g_n(y',t)| | \int_{\Rn}  \Ga'(x'-z', s) D_{z_i} E(y' -z',0) dz'|dy'ds\\
& := J_1 + J_2.
\end{align*}

First, we estimate $J_1$. Note that $ |g_n(y',t) - g_n(y',s)| \leq
\om(g_n)(t-s, t)$. Hence, using the lemma \ref{120907-2}, if $t \leq
M^2$, we have
\begin{align}\label{changeofvariables}
\notag J_1 & \leq
       C\int_0^t \om(g_n)(s, t) s^{-\frac32} x_n e^{-\frac{x_n^2}{s}} ( 1 + a(M,s))ds\\
& \notag \leq C \int_0^t \om(g_n)(s, t) s^{-1} ( 1 + |\ln M| + |\ln
\, s |)ds \\
& \leq C_M (\| g_n\|_{L^\infty} + \| g_n\|_{logDini}).
\end{align}
Here, we used the fact that for any $m> 0$,  $e^{-a} \leq c_m
a^{-m}$ for the second inequality. If $t \geq M^2$, then we have
\begin{align}\label{changeofvariables-2}
\notag J_1& \leq C\int_0^{M^2} \om(g_n)(s, t) s^{-1} ( 1 + |\ln M| +
            |\ln \, s |)ds \\
 \notag      & + C \|g_n \|_{L^\infty} \int_{M^2}^t   s^{-\frac32} x_n e^{-\frac{x_n^2}{s}} ( 1 +|\ln M|  )ds\\
       & \leq C_M (\| g_n\|_{L^\infty} + \| g_n\|_{logDini}).
\end{align}
Hence, by \eqref{changeofvariables} and \eqref{changeofvariables-2}, we get
\begin{align}\label{0304J-1}
J_1\leq C_M (\| g_n\|_{L^\infty} + \| g_n\|_{logDini}).
\end{align}

Next, we estimate $J_2$.
If $t \geq  M^2$, then
\begin{align}\label{J-2-1}
 \notag J_2 &  \leq C\|g_n\|_{L^\infty } \int_t^\infty s^{-\frac32} x_n e^{-\frac{x_n^2}s}
        ( 1 + |\ln M| )ds\\
 \notag &  = C_M\|g_n\|_{L^\infty } \int_t^\infty s^{-\frac32} x_n e^{-\frac{x_n^2}s} ds\\
&  \leq C_M \|g_n\|_{L^\infty}.
\end{align}

If  $t \leq M^2$, then, we get
\begin{align}\label{J-2-2}
J_2 & \leq C  \|g_n\|_{L^\infty} \int_{M^2}^\infty s^{-\frac32} x_n
e^{-\frac{x_n^2}s} ( 1 + |\ln M| )ds\\
& \notag  \qquad + C |g_n(t)| \int_t^{M^2} s^{-\frac32} x_n
e^{-\frac{x_n^2}s}
        ( 1 + \ln M + |\ln \, s | )ds \\
&  \notag  \leq C_M \Big( \|g_n\|_{L^\infty }  +  |g_n(t)|
\int_t^{M^2} s^{-1} |\ln{s} |ds \Big).
\end{align}
Note that by compatibility condition, we have $g_n(0) =0$. Hence,
for $t \leq s$, we have  $|g_n(t) | = |g_n(t) - g_n(0)| \leq
w(g_n)(t,t) \leq w(g_n)(s,t)$. Hence,  we get
\begin{align}\label{J-2-3}
J_2 \leq C \Big( \|g_n\|_{L^\infty }  +\|  g_n \|_{logDini}  \Big).
\end{align}

Hence, by from \eqref{J-2-1} to \eqref{J-2-3}, we get
\begin{align} \label{J-2}
J_2 \leq C\Big( \|g_n\|_{L^\infty }  +\|  g_n \|_{logDini}  \Big)
\end{align}
for all $x \in {\bf R}^n_+$ and $t> 0$.
Hence,  by \eqref{0304J-1} and \eqref{J-2}, we have that for
$x' \in {\bf R}^n_+$ and $0< t  $, we have
\begin{align*}
& |{\mathcal B}_{in}g_n(x,t)
             - \frac{\pa}{\pa x_i} {\mathcal E}g_n(x,t)|\leq C\Big( \|g_n\|_{L^\infty }  +\|  g_n \|_{logDini}  \Big).
\end{align*}
Hence, we completed the proof of lemma \ref{lemma0911}.

\section{Proof of Theorem \ref{s-t-theorem}}
\setcounter{equation}{0}
For the interior $L^\infty$ bound estimate, the authors in \cite{CC}
showed the boundedness using the layer potential method in
\cite{Sh1}.
\begin{prop}
Suppose the boundary data $g\in L^\infty(0,T;L^2(\partial\Om))$. If $dist(x,\partial\Om) \geq r_0 >0$,
$\epsilon >0$ and $t<T$,
then there is $C$ such that
$$
|u(x,t)| \leq \frac{C}{ r_0^{n-1-\epsilon}} ||g||_{L^\infty(0,T;L^2(\partial\Om))}.
$$
\end{prop}
(See Corollary 4.2 in \cite{CC}).

Let $x \in \Om$. Assume $dist(x, \pa \Om) < r_0$ for some fixed  small $r_0>0$.
Since Stokes equations is translation and rotation invariant, we assume that $P=0$ and
$x =(0,x_n),\,\, x_n >0$.
If $x$ is close enough to $\partial\Om$, there is a ball $B_{r_0}(0)$ centered at origin and $C^2$ function
 $\Phi :{\bf R}^{n-1} \rightarrow {\bf R}$ such that $\Om\cap B_{r_0}(0) = \{ y_n > \Phi(y')\} \cap B_{r_0}(0)$ and
 $\pa \Om\cap B_{r_0}(0) = \{ y_n = \Phi(y^{\prime})\} \cap B_{r_0}(0)$. Furthermore, $\Phi$ satisfies that
\begin{align}\label{flatboundary}
|\Phi(y')| \leq C |y'|^2,\quad
|\na^{\prime}\Phi (y')|\leq C|y'|, \quad
| \na^{\prime}\na^{\prime}\Phi (y^{\prime})|  \leq C
\end{align}
for $ y^\prime\in B^{'}_{r_0}(0) =\{ y' \in \Rn \, | \, |y'| < r_0 \}$ and the outward unit normal vector $N(Q)$ at
$Q= (y^\prime,\Phi(y^\prime))\in \pa \Om \cap B_{r_0}(0)$ is
$$
N(y^{\prime},\Phi(y^\prime)) = \frac{1}{ \sqrt{1+|\nabla^{\prime}\Phi(y^\prime)|^2}} (\nabla^\prime \Phi(y^\prime), -1).
$$
In particular, $N(x', \Phi(x')) = N(0,0) = (0, -1)$.
 Hence,  %the normal part of $\na {\bf S}(g \cdot N) (x,t) + \na {\bf T}(g \cdot N) (x,t)$ is
% $$
%\na {\bf S}(g \cdot N)_N (x,t) + \na {\bf T}(g \cdot N)_N (x,t) =  \sum_{1\leq k \leq n}
% \frac{\pa}{\pa x_k} {\bf S}(g \cdot N) N_k N_i
%-   \sum_{1\leq k \leq n}
% \frac{\pa}{\pa x_k} {\bf T} (g \cdot N) N_k N_i.
%$$
the $i$-th component of $\na {\bf S}(g \cdot N)_T (x,t) + \na {\bf T}(g \cdot N)_T (x,t)$
   is
\begin{align*}
& \frac{\pa}{\pa x_i} {\bf S}(g \cdot N) +\frac{\pa}{\pa x_i} {\bf T} (g \cdot N)
-  \sum_{1\leq k \leq n}
 \frac{\pa}{\pa x_k} {\bf S}(g \cdot N) N_k N_i
-   \sum_{1\leq k \leq n}
 \frac{\pa}{\pa x_k} {\bf T} (g \cdot N) N_k N_i\\
 & = \left\{\begin{array}{ll}\vspace{2mm}
   \frac{\pa}{\pa x_i} {\bf S}(g \cdot N)
 + \frac{\pa}{\pa x_i} {\bf T} (g \cdot N)
 &\qquad 1 \leq i \leq n-1,\\
0 &\qquad i =n.
 \end{array}
 \right.
\end{align*}
Note that for $1 \leq i \leq n-1$, we have
\begin{align*}
&\frac{\pa}{\pa x_i} {\bf S}(g \cdot N)(x,t)\\
& =   \int_{\pa \Om \cap B_{r_0}}
     \frac{\pa}{\pa x_i} E(x-y) ( g \cdot N )(y,t) d\si(y)
     +  \int_{\pa \Om \cap B^c_{r_0}}
     \frac{\pa}{\pa x_i} E(x-y) (g \cdot N) (y,t) d\si(y)\\
& : = A_{i1} + A_{i2}.
\end{align*}
Similarly, we have
\begin{align*}
&\frac{\pa}{\pa x_i}{\bf T}(g \cdot N)(x,t)\\
& = 4 \int_{0}^{t }\int_{ \pa\Om \cap B_{r_0}}   \frac{\pa}{\pa x_i} \ka(x-y,t-s)  (g \cdot N)(y,s)  d\si(y) ds\\
& \qquad       + 4 \int_{0}^{t }\int_{ \pa\Om \cap B^c_{r_0}} D_{x_i} \ka(x-y,t-s)
     (g \cdot N)(y,s)   d\si(y) ds\\
&:= B_{i1} + B_{i2}.
\end{align*}
Since $|x-y| \geq r_0$ for $y\in  \pa \Om \cap B^c_{r_0}$, we get
\begin{align*}
|A_{i2}|, \,\, |B_{i2}|   \leq C \| g\cdot N\|_{L^\infty(\pa \Om \times (0,T))}.
\end{align*}
Let $G(y',t):=  (g \cdot N)(y', \Phi(y'),t)\sqrt{1 + |\na \Phi(y')|^2}  $. Then, from $A_{i1}$ and $ B_{i1}$, we have
\begin{align*}
A_{i1} + B_{i1}
 & = \int_{  B^{'}_{r_0}} \frac{\pa }{\pa x_i}
   E( y', x_n - \Phi(y')) G(y',t) dy'\\
  &  \qquad +   4\int_0^t \int_{  B^{'}_{r_0}} \frac{\pa }{\pa x_i}
   \ka( y', x_n - \Phi(y'),t-s) G(y',s) dy' ds\\
& : = C_{i11} + C_{i12} + C_{i13},
\end{align*}
where
\begin{align*}
  C_{i11} :& = \int_{  B^{'}_{r_0}}
\Big( \frac{\pa }{\pa x_i}E( y', x_n - \Phi(y')) -\frac{\pa }{\pa x_i}E( y', x_n )\Big)
  G(y',t) dy'\\
C_{i12}: &=  4\int_0^t \int_{  B^{'}_{r_0}} \Big( \frac{\pa }{\pa x_i}
   \ka( y', x_n - \Phi(y'),t-s) - \frac{\pa }{\pa x_i}
   \ka( y', x_n,t-s) \Big) G(y',t) dy' ds\\
C_{i13} : &=  \int_{  B^{'}_{r_0}}
\frac{\pa }{\pa x_i}E( y', x_n ) G(y',t) dy'\\
& \quad   + 4\int_0^t \int_{  B^{'}_{r_0}} \Big( \frac{\pa }{\pa x_i}
   \ka( y', x_n,t-s) \Big) G(y',s) dy' ds.
\end{align*}
Note that $C_{i13}$ is the $i$-th component of $\nabla {\bf S}(G)_T (x,t)
+\nabla {\bf T}(G)_T (x,t)  $. Hence, by the lemma \ref{lemma0911}, we get
\begin{align}\label{130205-3}
|C_{i13}| \leq C (\| G\|_{L^\infty } + \| G \|_{logDini}  ).
\end{align}

Since $|\Phi(y')|\leq c |y'|^2$ for $y' \in B'_r$, using the Mean-value theorem, we have
\begin{align*}
|C_{i11}|&\leq  C\int_{  B^{'}_{r_0}}
      \frac{1}{|y'|^n} |\Phi(y')|
      |G(y',t)| dy'\\
& \leq C  \int_{  B^{'}_{r_0}} \frac{1}{|y'|^{n-2}} |G(y',t)| dy' \\
& \leq C\|G\|_{L^\infty(B^{'}_{r_0} \times (0,T))}
\end{align*}
and, similarly, we have
\begin{align*}
|C_{i12}| & \leq 4 \int_{0}^{t }\int_{ B^{'}_{r_0}} |\frac{\pa}{\pa x_i} \ka(y', x_n - \Phi(y'),t-s)
    - \frac{\pa}{\pa x_i} \ka(y', x_n ,t-s)|
     |G(y',s)| dy'ds\\
 & \leq C\| G\|_{  L^\infty(B^{'}_{r_0} \times (0,T))}.
\end{align*}
Since $\| G\|_{L^\infty (B^{'}_{r_0} \times (0,T))} \leq C \| g \cdot N\|_{L^\infty (\pa \Om \cap B^{'}_{r_0} \times (0,T))}$ and
$\| G\|_{logDini} \leq C( \| g \cdot N\|_{L^\infty} +  \| g \cdot N\|_{logDini})$, we get
\begin{align*}
|\na {\bf S}(g \cdot N)_T (x,t) + \na {\bf T}(g \cdot N)_T(x,t) | \leq  C( \| g \cdot N\|_{L^\infty} +  \| g \cdot N\|_{logDini}).
\end{align*}
Therefore, we completed the proof of our main Theorem \ref{s-t-theorem}.

\noindent
$\Box$

\section{Proof of Theorem \ref{c-example}}

For the simplicity, we assume that $n =2$.
Let $g_1 =0$ and  $g_2(x_1,t) = \chi_{(-1, 1)}(x_1) \chi_{\frac12 < t < 1} (t)$.
We will show that for $x_n < \frac12$ there is a positive constant $c> 0$ such that
\begin{align}\label{small x-2}
u^1(1, x_2, 1+ x_2^2) \leq c( \ln x_n +1).
\end{align}
Hence, we proved the theorem \ref{c-example}.
By the representation \eqref{representation} of solution and  by the second identity of \eqref{1006-3}, we have
\begin{align*}
u^1(x,t) & =\int_0^t \int_{\Rn}
K_{12}( x_1-y_1,x_2,t-s)g_2(y_1,s) dy_1 ds\\
& = - \int_0^t \int_{{\mathbb R} }
L_{12}(x_1-y_1,x_2,t-s)g_2(y_1,s) dy_1 ds\\
& + \int_{{\mathbb R}}
D_{y_1} E(x_1-y_1,x_2 )g_2(y_1, t) dy_1\\
& = - \int_0^t \int_{{\mathbb R}}
L_{21}( x_1-y_1,x_2,t-s)g_2(y_1,s) dy_1ds\\
& -  \int_0^t \int_{{\mathbb R}}
B_{12}(x_1-y_1,x_2,t-s)g_2(y_1,s) dy_1ds\\
& + \int_{{\mathbb R}}
D_{y_1} E( x_1-y_1,x_2 )g_2 (y_1, t) dy_1: = u^1_1 + u_2^1 + u_3^1.
\end{align*}
By the   \eqref{0126-1}, $u^1_1$ is bounded in $\Rn \times (0, \infty)$  and   $u_3^1 =0$ for $t > 1$.
Hence, for $t > 1$, we get
\begin{align}\label{u-1-estimate}
u^1(x,t) \leq u^1_2(x,t) + c.
\end{align}
Note that
\begin{align*}
u^1_2(x,t) = \int_0^t
\int_{{\mathbb R}} D_2\Ga(x_1 -y_1, x_2,t-s) H g_2(y_1, s) dy_1 ds,
\end{align*}
where $H$ is Hilbert transform. By the direct calculation, we have
\begin{align*}
H g_2(x_1,t) &= \frac{1}{\pi}  \chi_{\frac12 < t < 1} (t) \int_{-1}^1 \frac{1}{x_1 -y_1} dy_1\\
&=\frac{1}{\pi} \chi_{\frac12 < t < 1} (t) \left\{\begin{array}{l}\vspace{2mm}
  \ln(x_1 +1) - \ln(1 -x_1), \quad  |x_1| <1\\ \vspace{2mm}
 \ln (1 +x_1) - \ln (x_1 -1),\quad x_1 > 1\\
\ln (-1 -x_1) - \ln (1 -x_1),\quad x_1 < -1.
 \end{array}
 \right.
\end{align*}
We take $x_1 = 1 $ and $t = 1 + x^2_2$ for small $x_2 > 0$. Then, we get
\begin{align}\label{I's}
u_2^1(1,x_2,1+ x_2^2) &= \int_{\frac12}^{1}
\int_{{\mathbb R}} D_2\Ga(1 -y_1, x_2,1 + x_2^2-s) H g_2(y_1, s) dy_1 ds\\
\notag  & = I_1 + I_2 + I_3,
\end{align}
where
\begin{align*}
I_1 : & = \int_{\frac12}^{1}
\int_{-\infty}^0 D_2\Ga(1 -y_1, x_2,1 + x_2^2-s) H g_2(y_1, s) dy_1 ds,\\
I_2 : & = \int_{\frac12}^{1}
\int_0^2 D_2\Ga(1 -y_1, x_2,1+ x^2_2-s) H g_2(y_1, s) dy_1 ds,\\
I_3 : & = \int_{\frac12}^{1}
\int_2^\infty D_2\Ga(1-y_1, x_2,1 + x^2_2-s) H g_2(y_1, s) dy_1 ds.
\end{align*}
Since $Hg_2$ is bounded in the intervals  $(-\infty, 0)$ and $ (2, \infty)$, $I_1$ and $I_3$ are bounded. And
\begin{align*}
I_2 &  = \frac{1}{\pi} \int_{\frac12}^{1}
\int_0^1  D_2\Ga(1 -y_1, x_2,1+ x_2^2-s) ( \ln(y_1 +1) - \ln(1 -y_1)) dy_1 ds\\
& \quad         +  \frac{1}{\pi}  \int_{\frac12}^{1}
\int_1^2  D_2\Ga(1 -y_1, x_2,1 + x_2^2-s) ( \ln (1 +y_1) - \ln (y_1 -1) ) dy_1 ds \\
& = \frac{1}{\pi} \int_{\frac12}^{1}
\int_0^1  D_2\Ga(1-y_1, x_2,1 + x_2^2-s)  ( - \ln(1 -y_1)) dy_1 ds\\
&\quad         +  \frac{1}{\pi}  \int_{\frac12}^{1}
\int_1^2  D_2\Ga(1 -y_1, x_2,1 +x_2^2-s) (   - \ln (y_1 -1) ) dy_1 ds + Bd,
\end{align*}
where $Bd$ is bounded such that $|Bd| \leq c$ for some positive constant $c> 0$. Hence, from \eqref{u-1-estimate} and \eqref{I's}, we get
\begin{align*}
u^1(1, x_2, 1 + x_2^2) & \leq I_2 + c\\
& \leq   \frac{1}{\pi}\int_{\frac12}^{1}
\int_0^1  D_2\Ga(1 -y_1, x_2,1 + x_2^2-s)  ( - \ln(1 -y_1)) dy_1 ds + c\\
& =   \frac{1}{2\pi^2}\int_{\frac12}^{1}
\int_0^1  \frac{x_2}{(1 + x_2^2-s)^2} e^{-\frac{x_2^2}{2(1 + x_2^2-s)}} e^{-\frac{|1 -y_1|^2}{1 + x_2^2-s}}    \ln(1 -y_1) dy_1 ds + c \\
& \leq    \frac{1}{2\pi^2} e^{-\frac12} \int_{\frac12}^{1}
\int_{1 -  x_2}^1  \frac{x_2}{(t-s)^2} e^{-\frac{x_2^2}{2(1 + x_2^2-s)}}   \ln(1 -y_1) dy_1 ds + c\\
& =    \frac{1}{2\pi^2} e^{-\frac12} \int_{\frac12}^{1}
   \frac{x_2}{(1 + x_2^2-s)^2} e^{-\frac{x_2^2}{2(1 + x_2^2-s)}}  x_2 (\ln(x_2) +1)   ds + c\\
& =    \frac{1}{2\pi^2} e^{-\frac12} \int^{1 + x_2^2-\frac12}_{x_2^2}
   \frac{x_2}{s^2} e^{-\frac{x_2^2}{s}} x_2 (\ln(x_2) +1) ds +c\\
& =    \frac{1}{\pi^2} e^{-\frac12} x_2 (\ln(x_2) +1) \int_{\frac{x_2^2}{\frac12 + x_2^2}}^1
     e^{-s}    ds +c.
\end{align*}
This implies  \eqref{small x-2}.
$\Box$

\noindent
{\bf Acknowledgment.} The research is supported by KRF2011-028951.

\end{document}